\documentclass{article}
 \usepackage[latin1]{inputenc}
\usepackage[T1]{fontenc}
\usepackage[english]{babel}
\usepackage{psfrag}
\usepackage{amsmath}
\usepackage{lmodern}
\usepackage{fancyhdr}
\usepackage{graphicx}
\usepackage{epsfig}
\usepackage{epstopdf}
\usepackage{pifont}
\usepackage{amsmath}
\usepackage{amsthm,amssymb}
\newtheorem{theoreme}{Theorem}[section]
\newtheorem{proposition}{Proposition}[section]
\newtheorem{definition}{Definition}[section]
\newtheorem{corollaire}{Corollaire}[section]

\theoremstyle{remark}
\newtheorem{lemme}{Lemma}[section]
\DeclareGraphicsExtensions{.eps}
\DeclareGraphicsExtensions{.jpeg}
\def\A{\mathcal A}
\date{\today}
\begin{document}

\begin{center}

\textbf{Common dynamics of two Pisot substitutions with the same incidence matrix
} \vspace{0.5cm}

Tarek Sellami\footnote{sellami@iml.univ-mrs.fr\\ Institut de Math\'ematiques de
Luminy, CNRS U.M.R. 6206, 163, Avenue de Luminy, Case 907, 13288 Marseille
Cedex 09 France. \\D\'epartement de Math\'ematiques, Facult\'e des Sciences de
Sfax, BP 802, 3018 Sfax, Tunisie.}

\end{center}

\vspace{0.5cm} \textbf{Abstract}. The matrix of a substitution is not
sufficient to completely determine the dynamics associated, even in
simplest cases since there are many words with the same abelianization.\\
In this paper we study the common points of the canonical broken lines
associated to two different Pisot irreducible substitutions $\sigma_1$ and
$\sigma_2$ having the same incidence matrix. We prove that if $0$ is inner
point to the Rauzy fractal associated to $\sigma_1$  these common points can be generated with a substitution on an alphabet
of so-called "balanced blocks".

\vspace{1cm} \textbf{R\'esum\'e}: On sait que la matrice d'une substitution ne
suffit pas \`a d\'eterminer compl\`etement le syst\`eme dynamique associ\'e;
m\^eme dans les cas les plus simples, il existe de nombreuses substitutions
associ\'ees \`a une matrice : il existe de nombreux mots ayant le m\^eme
ab\'elianis\'e.\\ Dans ce papier, on  \'etudie les points commun de deux lignes
bris\'ees associ\'e aux deux substitutions $\sigma_1$ et $\sigma_2$ irreducible
de type Pisot qui ont la m\^eme matrice d'incidence. On montre que si $0$ est
un point int\'erieur \`a l'un des deux fractals de Rauzy associ\'e \`a
$\sigma_1$ ou $\sigma_2$ alors ces points communs peut \^etre g\'en\'er\'es par
une substitution d\'efinit sur un alphabet appel\'e "block balnc\'e". \vspace{1cm}
\section{Introduction}

Let $\sigma_1$ and $\sigma_2$ be two different Pisot substitutions having the
same incidence matrix. Although  the fixed points of each substitution have the
same letter frequencies, they usually show different dynamical and geometrical
properties, e.g., their Rauzy fractals have different properties. (The Rauzy
fractals can give a geometric model of the dynamical system defined by the
substitution, for more detail see section $2$).

\vspace{0.3cm}
A classic example is given by  the Tribonacci
substitution and the flipped Tribonacci substitution, i.e.,

\begin{center}
$\sigma_1:\left\{
\begin{array}{ll}
a\rightarrow ab\\
b\rightarrow ac\\
c\rightarrow a
\end{array}\right.$
\hspace{1cm}and \hspace{1cm} $\sigma_2:\left\{
\begin{array}{ll}
a\rightarrow ab\\
b\rightarrow ca\\
c\rightarrow a
\end{array}\right.$
\end{center}
\vspace{1cm} The incidence matrix of $\sigma_1$ and $\sigma_2$ is
$\begin{pmatrix}
1 & 1 & 1 \\
1 & 0 &0 \\
0 & 1 & 0 \\
\end{pmatrix}
$. The dominant eigenvalue satisfies the relation $X^3-X^2-X-1=0$, hence the
name Tribonacci for the substitution.

 The Rauzy fractal of the first substitution is a topological disc \cite{PA},
simply connected , while it is a well known fact that the second fractal is not
simply connected, compare Figure[1]. \vspace{0.5cm}
\begin{figure}[h]
\begin{center}
\label{compare_fractals}
\scalebox{0.4}{\includegraphics{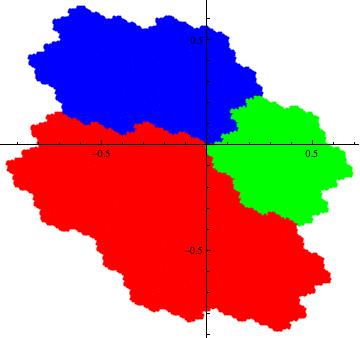}}
\vspace{0.5cm}
\scalebox{0.4}{\includegraphics{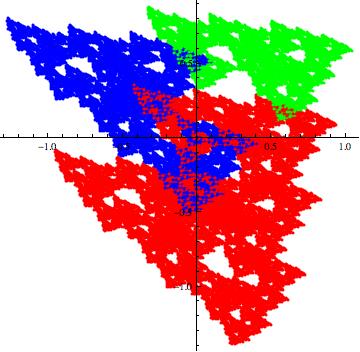}}
\caption{ The Rauzy  fractals of  $\sigma_1$ and $\sigma_2$}
\end{center}
\end{figure}

\vspace{0.5cm}

We consider another simple example of substitutions $\tau_1$ and $\tau_2$,
i.e.,

\begin{center}
$\tau_1:\left\{
\begin{array}{ll}
a\rightarrow aba\\
b\rightarrow ab\\
\end{array}\right.$
\hspace{1cm} and\hspace{1cm} $\tau_2:\left\{
\begin{array}{ll}
a\rightarrow aab\\
b\rightarrow ba\\
\end{array}\right.$
\end{center}

The Rauzy fractal of $\tau_2$ is the closure of a countable union of disjoint
intervals and the Rauzy fractal of $\tau_1$ is an interval, see \cite{AM} and Figure[6].\\
We can deduce from one matrix we can obtain many different substitutions, so
many different Rauzy fractals. We are interested to studies commons dynamics of
these Rauzy fractals, we are  interested to characterize their intersection, for this we need to define a new object.
prove that we can consider their intersection as a substitutive set.

\begin{definition}
A substitutive set is the closure of the projection of a canonical stepped line associated to a primitive substitution on a contracting space associated to the restriction of a positive integer matrix. For more detail see section 2.
\end{definition}

The main result of this paper is the following:
\begin{theoreme}

Let $\sigma_1$ and $\sigma_2$ be two irreducible unimodular Pisot substitutions
with the same incidence matrix. Let $X_{\sigma_1}$ and $X_{\sigma_2}$ the  two
associated  Rauzy fractals;
suppose that $0$  is inner point to $X_{\sigma_1}$  .\\
Then the intersection  of $X_{\sigma_1}$ and $X_{\sigma_2}$ has non-empty
interior, and it is substitutive.There is an  algorithm to obtain the substitution for intersection.
\end{theoreme}

\section{Substitutions and Rauzy fractals}

\subsection{General setting}

Let $\mathcal{A}:=\{a_1,...,a_d\}$ be a finite set of cardinal $d$ called
alphabet. The free monoid  $\mathcal{A}^*$ on the alphabet $\mathcal{A}$ with
empty word $\varepsilon$ is defined as the set of finite words on the alphabet
$\mathcal{A}$, this is $\mathcal{A}^*:=\bigcup_{k\in\mathbb{N}}\mathcal{A}^k$,
endowed with the concatenation map. We denote by $\mathcal{A}^{\mathbb{N}}$ and
$\mathcal{A}^{\mathbb{Z}}$ the set of one and two-sided sequences on
$\mathcal{A}$, respectively. The topology of $\mathcal{A}^{\mathbb{N}}$ and
$\mathcal{A}^{\mathbb{Z}}$ is the product topology of discrete topology on each
copy of $\mathcal{A}$. Both spaces are metrizable.

The length of a word $w \in \mathcal{A}^n$ with $n \in \mathbb{N}$ is defined
as $\vert w\vert = n$. For any letter $a\in\mathcal{A}$, we define the number
of occurrences of $a$ in $w = w_1w_2 \ldots w_{n-1}w_n$ by $\vert w\vert_a
=\sharp \{i  | w_i = a\}$.

Let $l : \mathcal{A}^* \mapsto \mathbb{Z}^d : w\mapsto(\vert w\vert
_a)_{a\in\mathcal{A}}\in\mathbb{N}^d$ be the natural homomorphism obtained by
abelianization of the free monoid, called the abelianization map.

A substitution over the alphabet $\mathcal{A}$ is an endomorphism of the free
monoid $\mathcal{A}^*$ such that the image of each letter of $\mathcal{A}$ is a
nonempty word.\\A substitution $\sigma$ is primitive if there exists an integer
$k$ such that, for each pair $(a, b)\in \mathcal{A}^2$, $\vert \sigma^k(a)
\vert_b > 0$. We will always suppose that the substitution is primitive, this
implies that for all letter $j \in \A$ the length of the successive iterations
$\sigma^k(j)$ tends to infinity.

A substitution naturally extends to the set of two sided sequences
$\mathcal{A}^{\mathbb{Z}}$. We associate to every substitution $\sigma$ its
incidence matrix $M$ which is the $n\times n $ matrix obtained by
abelianization, i.e. $M_{i,j} = \vert \sigma(j) \vert_i$. It holds that
$l(\sigma(w))=Ml(w)$ for all $w\in\mathcal{A}^*$.\\

\textbf{Remark.} The incidence matrix of a primitive substitution is a
primitive matrix, so with the Perron-Frobenius theorem, it has a simple real
positive dominant eigenvalue $\beta$.

\subsection{Rauzy fractals}
\begin{definition}
A Pisot number is an algebric integer $\beta>1$ such that each Galois
conjugate $\beta^{(i)}$ of $\beta$ satisfies $\mid\beta^{(i)}\mid<1$.

\end{definition}

From now, we will suppose that all the substitutions that we consider are
irreducible of Pisot type and unimodular. This mean that the characteristic
polynomial of its incidence matrix is irreducible, its determinant is equal to
$\pm1$ and its dominant eigenvalues is a Pisot number. We can prove that any
irreducible Pisot substitution is primitive (see  \cite{JML}).\\

\textbf{Remark.} Note that there exist substitution whose largest eigenvalue
is Pisot but whose incidence matrix has eigenvalues that are not conjugate to
the dominant eigenvalue. Example is 1$\rightarrow$ 12, 2$\rightarrow$ 3,
3$\rightarrow$ 4, 4$\rightarrow$ 5, 5$\rightarrow $1. The characteristic
polynomial is reducible. Such substitutions are called Pisot reducible.

\begin{definition}
Let $\sigma$ a substitution and $u\in \mathcal{A}^\mathbb{N}$, $u$ is a fixed
point of $\sigma$ if $\sigma(u)=u$. The infinite word $u$ is a periodic point
of $\sigma$ if there exist $k\in\mathbb{N}$ such that $\sigma^k(u)=u$.
\end{definition}
Let $\sigma$ be a primitive substitution, then there exist a finite number of
periodic points (see \cite{PF}). We associate to the fixed point $u$ of the
substitution a symbolic dynamical system $(\Omega_u, S)$
 where $S$ is the shift map on $\mathcal{A}^\mathbb{N}$ given by $S(a_0a_1...)=a_1a_2...$
 and $\Omega_u$ is the closure of $\{S^m(u): m\geq0\}$ in
 $\mathcal{A}^\mathbb{N}$.\\

\textbf{Remark.} If $\sigma$ is a primitive substitution then the symbolic
dynamical system $(\Omega_u, S)$ does not depend on $u$; we denote it by
$(\Omega_\sigma,S$).

We say that a dynamical system $(X, f)$ is semiconjugate to another dynamical
system $(Y, g)$ if there exists a continuous surjective map $\Theta :
X\rightarrow Y$ such that $\Theta\circ f=g\circ\Theta$. An important question
is whether and how the symbolic dynamical system $(\Omega_u, S)$ admit a
geometric model. By geometrically realizable we mean there exists a dynamical
system $(X, f)$ defined on a geometrical structure, such that $(\Omega_u, S)$
is semiconjugate to $(X, f)$.

In \cite{GR}, G.Rauzy proves that the dynamical system generated by the
substitution $\sigma(1)=12$, $\sigma(2)=13$, $\sigma(3)=1$, is
measure-theoretically conjugate to an exchange of domains in a compact set
$\mathcal{R}$ of the complex plane. This compact subset has a self-similar
structure : using method introduced by F.M.Dekking in \cite{FMD}, S.Ito and
Pierre Arnoux obtain in \cite{PA} an alternative construction of $\mathcal{R}$
and prove that each of exchanged domains has fractal boundary. We will use the
projection method to obtain the Rauzy fractal.

\begin{definition}
A stepped line $L=(x_n)$ in $\mathbb{R}^d$ is a sequence (it could be finite or infinite) of points in $\mathbb{R}^d$
such that $x_{n+1}-x_n$ belong to a finite set.\\
A canonical stepped line is a stepped line such that $x_0=0$ and for all $n\geq 0$, $x_{n+1}-x_n$ belong to
the canonical basis of $\mathbb{R}^d$.
\end{definition}
Using the abelianization map, to any finite or infinite word $W$, we can
associate a canonical stepped line in  $\mathbb{R}^d$ as a sequence $(l(P_n))$,
where $P_n$ are the prefix of length $n$ of $W$.

An interesting property of the canonical stepped line associated to a fixed point of primitive Pisot
substitution is that it remains within bounded
distance from the expanding direction (given by the right eigenvector of
Perron-Frobenius of $M_{\sigma}$).
We denote by $E_s$ the stable space (or contracting space) and $E_u$ the
unstable space (or expanding direction).
We denote by $\pi_s$ the linear projection in the contracting plane, parallel to the
expanding direction and $\pi_u$ the projection in the
expanding direction parallel to the contracting plane.
We will project the stepped line on the contracting space in the direction of the right Perron-Frobenius
eigenvector. We obtain a bounded set in $(d-1)$-dimensional vector space.\\

\begin{definition}

Let $\sigma$ an irreducible Pisot substitution. The Rauzy fractal associated to
$\sigma$ is the closure of the projection of the canonical stepped line
associated to any fixed point of $\sigma$ in the contracting plane parallel to
the expanding direction.
\end{definition}
We note the projection $\pi$ of the orbit of the fixed point associated to a Pisot irreducible substitution $\sigma$ on a contracting space associated to its incidence matrix.

\begin{proposition}
The projection $\pi$ of the symbolic dynamical system $\Omega_\sigma$
associated to a Pisot irreducible substitution $\sigma$ to the Rauzy fractal is
a continuous map.
\end{proposition}
\begin{proof}
The proof is given in \cite{PF}.
\end{proof}
 We denote by $X_\sigma$ the Rauzy fractal
(Central tile) associated to $\sigma$ : $ X_\sigma :=
\overline{\{\pi_s(l(u_0\ldots u_{k-1}), k\in{\mathbb{N}}\}} $. with $u_0\ldots
u_{k-1}$ is a prefix of the fixed point of length $k$. Subtiles of the central
tile $X_\sigma$ are naturally defined, depending on the letter associated to
the vertex of the stepped line that is projected. On thus sets for
$i\in\mathcal{A}$ : $ X_\sigma(i) := \overline{\{\pi_s(l(u_0\ldots u_{k-1}),
k\in{\mathbb{N}}, u_k=i\}} $.

\begin{proposition}
Let $\sigma$ a Pisot substitution and $X_\sigma$ its associated Rauzy fractal.
The boundary of $X_\sigma$ has zero measure.
\end{proposition}
\begin{proof}
See \cite{MB} and \cite{VS}.
\end{proof}

\begin{figure}[h]
\begin{center}
\scalebox{0.5}{\includegraphics{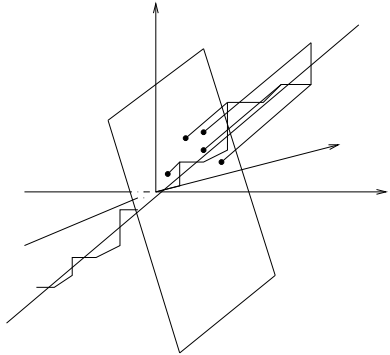}}

\caption{ The projection method to get the Rauzy fractal.}
\end{center}
\end{figure}

\subsection{Central tiles viewed as a graph directed iterated function}

The tiles $X_\sigma(i)$ can be written as a so-called graph iterated function
system (GIFS).

\begin{definition}(GIFS)

Let G be a finite directed graph with set of vertices $\{1,\ldots,q\}$ and set
of edges $E$. Denote the set of edges leading from $i$ to $j$ by $E_{ij}$. To
each $e\in E$ associated a contractive mapping $\tau_e:\mathbb{R}^n\rightarrow
\mathbb{R}^n$. If for each $i$ there is some outgoing edge we call
$(G,\{\tau_e\})$ a GIFS.
\end{definition}

\begin{definition}(Prefix-suffix automaton)

Let $\sigma$ be a substitution over the alphabet $\mathcal{A}$ and let
$\mathcal{P}$ be the finite set
\begin{center}
${P}:=
\{(p,a,s)\in\mathcal{A}^\ast\times\mathcal{A}\times\mathcal{A}^\ast; \exists
b\in \mathcal{A}, \sigma(b)=pas\}$.
\end{center}
The prefix-suffix automaton of sigma has $\mathcal{A}$ as a set of vertices and
$\mathcal{P}$ as a set of label edges : there is an edge labeled by $(p,a,s)$
from $a$ to $b$ if and only if $pas=\sigma(b)$.
\end{definition}

\textbf{Example.} For the Fibonacci substitution $1\mapsto 12$ and $2\mapsto
1$, one gets:
\begin{center}
${P}=\{(e,1,2),(1,2,e),(e,1,e)\}$.
\end{center}

The prefix-suffix automaton of the Fibonacci substitution is :

\begin{figure}[h]
\begin{center}

\scalebox{0.4}{\includegraphics{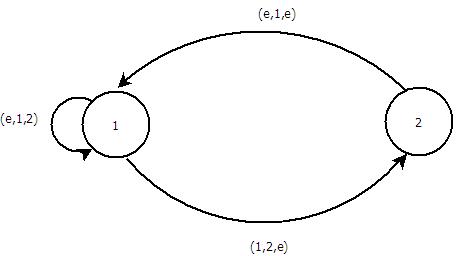}}

\caption{ Prefix-suffix automaton for the fibonacci substitution.}
\end{center}
\end{figure}

\begin{theoreme}
Let $\sigma$ be a primitive unit Pisot substitution over the alphabet
$\mathcal{A}$. The central tile $X_\sigma$ is a compact subset with nonempty
interior. Each subtile is the closure of its interior. The subtiles of
$X_\sigma$ are solution of the $GIFS$
\begin{center}
$\forall i \in \mathcal{A}, X_\sigma(i) = \bigcup_{j\in\mathcal{A}, i
\xrightarrow [\text{}]{\text{(p,i,s)}}j}$   $M X_\sigma (j) + \pi_s l(p)$.

\end{center}
\end{theoreme}
\begin{proof}
The proof is given in \cite{ASJT}.

\end{proof}

\subsection{Disjointness of the subtiles of the central tile}

To ensure that the subtiles are disjoint, we introduce the following
combinatorial condition in substitutions.

\begin{definition}(Strong coincidence condition).

A substitution $\sigma$ over the alphabet $\mathcal{A}$ satisfies the strong
coincidence condition if for every pair $(b_1,b_2)\in \mathcal{A}^2$, there
exist $k\in\mathbb{N}$ and $a\in\mathcal{A}$ such that $\sigma^k(b_1)=p_1as_1$
and $\sigma^k(b_2)=p_2as_2$ with $l(p_1)=l(p_2)$ or $l(s_1)=l(s_2)$.
\end{definition}
\textbf{Remark.} The strong coincidence condition is satisfied by every unit
Pisot substitution over two letter alphabet \cite{MBBD}. It is conjectured that
every substitution of Pisot type satisfies the strong coincidence condition.

\begin{theoreme}
Let $\sigma$ be a primitive unit Pisot substitution. If $\sigma$ satisfies the
strong coincidence condition, then the subtiles of the central tiles have
disjoint interiors.
\end{theoreme}

\begin{proof}
The proof for the disjointness is given in \cite{PA}.
\end{proof}
\textbf{Remark.} If $0$ is inner point to the Rauzy fractal associated to a
Pisot substitution then the subtiles of the central tiles have disjoint
interiors (see \cite{ASJT}).

\subsection{Substitutive sets}

A substitutive set is the closure of the projection of a canonical stepped line associated to substitution on a contracting space of a  restriction of a positive integer matrix. In particular a Rauzy fractal is a substitutive set since it is the projection of  canonical stepped line associated to a fixed point  on the contracting space associated to the matrix of substitution. So we can expand the definition of Rauzy fractal to substitutive set. 
In particular a substitutive set can be expressed as the attractor of some graph directed iterated function system (IFS). See \cite{PAJBXB}

\section{Intersection of Rauzy fractals}
Let $\sigma_1$ and $\sigma_2$ two Pisot irreducible substitutions with the same incidence matrix, we consider
$X_{\sigma_1}$ and $X_{\sigma_2}$ their associated Rauzy fractals respectively.\\
The intersection of  $X_{\sigma_1}$ and $X_{\sigma_2}$ is non-empty since it
contains $0$, and it is a compact set (intersection of two compacts).
\begin{proposition}

Let $\sigma_1$ and $\sigma_2$ be two Pisot irreducible substitutions with the same
incidence matrix. We consider $L_1$ and $L_2$ the canonical broken lines associated to a fixed point of $\sigma_1$ and
$\sigma_2$ respectively, let $P_1$ and $P_2$ two points from $L_1$ and $L_2$ respectively.
Then $\pi_s(P_1)=\pi_s(P_2)$ implies $P_1=P_2$.
\end{proposition}

\begin{proof}
The Perron Frobenius eigenvalues is irrational in the irreducible case. We
project $P_1$ and $P_2$ in the contracting space parallel to the expanding
space (the direction of the Perron Frobenius eigenvectors). If
$\pi_s(P_1)=\pi_s(P_2)$ then $(P_1P_2)$ is parallel to the expanding direction.
This implies that expanding direction is rational.

\end{proof}

\begin{proposition}
Let $\sigma_1$ and $\sigma_2$ be two substitutions with the same incidence
matrix, we consider $X_{\sigma_1}$ (resp. $X_{\sigma_2}$) the Rauzy fractal
associated to $\sigma_1$ (resp. $\sigma_2$) and $X_\sigma$ the common point of $X_{\sigma_1}$ and $X_{\sigma_2}$.\\
Then the boundary of $X_\sigma$ is included in the union boundary of
$X_{\sigma_1}$ and $X_{\sigma_2}$ and  has zero measure.
\end{proposition}

\begin{proof}
Let $x$ a point from the boundary of $X_\sigma$. We suppose that $x$ is not on
the boundary of $X_{\sigma_1}$. Then there exist $r_1>0$ such that
$B(x,r_1)\subset X_{\sigma_1} $. If $x$ is not in the boundary of
$X_{\sigma_2}$ then there exist $r_2>0$ such that $B(x,r_2)\subset
X_{\sigma_2}$. Then there exist $r=\min(r_1, r_2)$ such that $B(x,r)\subset
X_{\sigma_1}\cap X_{\sigma_2}$, $x$ is in the boundary of $X_\sigma$ then $x$
is in the boundary of $X_{\sigma_2}$.
Then $\partial X_\sigma \subset \partial X_{\sigma_1}\cup \partial X_{\sigma_2}$.\\
Since $\partial X_{\sigma_1}$ and $\partial X_{\sigma_2}$ have zero measure then $\partial X_\sigma$
has zero measure.
\end{proof}

\begin{figure}[h]

\begin{center}
\scalebox{0.5}{\includegraphics{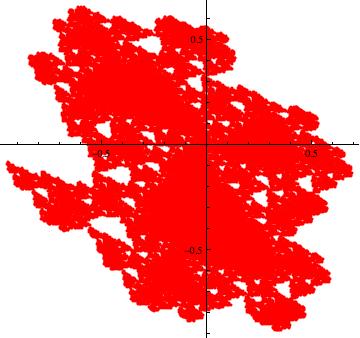}}

\caption{Sets of common points of the fractals of Tribonacci and the flipped
Tribonacci .}
\end{center}

\end{figure}

\subsection{The main result: Morphism generating the common points of two Pisot substitutions with the same incidence
matrix}

In this section  we consider $\sigma_1$ and $\sigma_2$ be two unimodular
irreducible Pisot substitutions with the same incidence matrix. We denote
$X_{\sigma_1}$ and $X_{\sigma_2}$ their associated  Rauzy fractals
respectively. We suppose that $0$ is an inner point to $X_{\sigma_1}$. We note
$X_\sigma$ the closure of the intersection of the interior of $X_{\sigma_1}$
and the interior of $X_{\sigma_2}$. Let $(\Omega_{\sigma_1},S)$ and
$(\Omega_{\sigma_2},S)$ the symbolic dynamical systems associated to $\sigma_1$
and $\sigma_2$ respectively. We consider $\pi_1$ (resp. $\pi_2$) the projection
map from the symbolic dynamical system $(X_{\sigma_1}, S)$ into the
Rauzy fractal (rep.$\pi_2$).\\

We will prove that $X_\sigma$ is a substitutive set, and it can be generated by
a substitution obtained with algorithm generating the common point of the
interior of $X_{\sigma_1}$ and $X_{\sigma_2}$.

\begin{definition}
For a dynamical system $(X,T)$ if $A$ is a subset of $X$, and $x\in A$, we
define the first return time of $x$ as $n_x = \inf \{n\in\mathbb{N}^*|T^n(x)\in
A\}$ (it is infinite if the orbit of $x$ does not come back to $A$). If the
first return time is finite for all $x\in A$, we define the induced map of $T$
on $A$ (or first return map) as the map $x\mapsto T^{n_k}(x)$, and we denote
this map by $T_A$.
\end{definition}

\begin{definition}
A sequence $u = (u_n)$ is minimal (or uniformly recurrent) if every word
occurring in $u$ occurs in an infinite number of positions with bounded gaps,
that is, if for every factor $W$, there exist $s$ such that for every $n$, $W$
is a factor of $u_n\ldots u_{n+s-1}$.

\end{definition}
\begin{lemme}
The closure of the intersection $X_\sigma$ has non empty interior and non-zero
Lebesque measure.
\end{lemme}
\begin{proof}
We suppose that $0$ is an inner point to $X_{\sigma_1}$. Then there exist an
open set $U$ such that $0\in U \subset X_{\sigma_1}$. The Rauzy fractal is the
closure of its interior and $0$  is a point of $X_{\sigma_2}$, hence there
exist a sequence of points $(x_n)_{n\in\mathbb{N}}$ from the interior of
$X_{\sigma_2}$ which converges to $0$. Then there exist open sets $V_n$ such
that $x_n\in V_n\subset X_{\sigma_2}$. Since $(x_n)$
converge to $0$, there exists $N\in\mathbb{N}$  such that $x_N\in U$.\\
We denote by $W$ the open set $W=U\cap V_N$,$W$ is non-empty and $W\subset
X_{\sigma_1}\cap X_{\sigma_2}$. The intersection of $X_{\sigma_1}$ and
$X_{\sigma_2}$ contains a non empty open set, hence it has non-zero Lebesque
mesure.
\end{proof}

We define the subgroup  $\Gamma$ of $\mathbb{Z}^d$  as : $$\Gamma =
\{\sum_{i=1}^d n_ie_i/ \sum_{i=1}^d n_i=0, n_i\in\mathbb{Z}\}$$ with $e_i$ is
the canonical bases of $\mathbb{R}^d.$

\begin{lemme}
Let $\sigma$ be an irreducible Pisot substitution, and $X_\sigma$ its
associated Rauzy fractal. If $0$ is inner point to $X_\sigma$ then $X_\sigma$
is a fundamental domain of $E_s$ for the projection of  $\Gamma$ on the stable
space.
\end{lemme}
\begin{proof}
The proof is given in \cite{ASJT}.

\end{proof}

\begin{lemme}
Let  $W$ be a non-empty open set in $X_\sigma$. Let $V_1 = \pi_1^{-1}(W)$ and
$V_2 = \pi_2^{-1}(W)$ from $\Omega_{\sigma_1}$ and $\Omega_{\sigma_2}$
respectively. If $n$ is a first return time in $V_2$ then $n$ is a return
time in $V_1$.

\end{lemme}
\begin{proof}

We consider $v_1\in V_1$ and $v_2\in V_2$ such that  $\pi_1(v_1)= \pi_2(v_2)$.
Let $n$ the first return time of $v_2$ in $V_2$. Then there exist $w\in\Gamma$
such that $\pi_1(S^n(v_1)) = \pi_2(S^n(v_2)) + \pi_2(w)$.
 $\pi_2(S^n(v_2))$ is a point from the interior of $X_\sigma$, then it is a
point from the interior of $X_{\sigma_1}$. And we have $\pi_1(S^n(v_1))$ is a
point from $X_{\sigma_1}$. \\
Since $0$ is an inner point to $X_{\sigma_1}$, from lemma $3.3$,
$X_{\sigma_1}$ is a fundamental domain. Then we obtain $\pi_2(w) = 0$. So we
have $\pi_1(S^n(v_1)) = \pi_2(S^n(v_2))$. Then if $n$ is a return time in $V_2$
we deduce that $n$ is a return time in $V_1$.
\end{proof}
\begin{definition}
Let $U$ and $V$ two finite words, we say that $\begin{pmatrix}
U\\
V\\
\end{pmatrix}$ is balanced block if $l(U) = l(V)$, where $l$ is the abelianization map from $\mathcal{A}^*$ in
$\mathbb{Z}^d$.
\end{definition}
\begin{definition}
A minimal balanced  block is a balanced block, such for every  strict prefix
$U_k$, $V_k$ of $U$ and $V$ respectively of length $k$, $l(U_k) \neq l(V_k)$.

\end{definition}
\begin{lemme}
Let $u$ and $v$ be tow fixed points of $\sigma_1$ and $\sigma_2$ respectively,
then we can decompose $u$ and $v$ on a finite minimal balanced blokcs.
\end{lemme}
\begin{proof}
Let $u$ and $v$ be tow fixed points of $\sigma_1$ and $\sigma_2$ respectively.
We have  $0\in X_\sigma$  then there exist $v_1$ and $v_2$ two prefix of $u$
and $v$ respectively such that $x = \pi(v_1) = \pi(v_2)$ and $l(v_1) = l(v_2)$.
We obtain a balanced block:  $\begin{pmatrix}
v_1\\
v_2\\
\end{pmatrix}$, we can decompose it with minimal balanced blocks and we consider the image of each new minimal
balanced block with $\sigma_1$ and $\sigma_2$.
 Then  there exist new minimal balanced blocks which  appear, we consider the image of each new blocks by
 $\sigma_1$ and $\sigma_2$. Since every word  appears with a bounded distance, all the minimal balanced blocks will appear after  a finite time. Then we can
obtain a decomposition of $u$ and $v$ with a finite number of minimal balanced
blocks. A simple case appears when $u$ and $v$ begin with the same letter $i$, 
then the first minimal  balanced block is $\begin{pmatrix}
i\\
i\\
\end{pmatrix}$.

\end{proof}

\begin{theoreme}

$X_{\sigma}$  is a substitutive set.
\end{theoreme}
\begin{proof}
We have $X_\sigma$ is the closure of the projection of points associated to
balanced blocks, from the two stepped lines associated to the fixed points of
$\sigma_1$ and $\sigma_2$. These common points can be obtained as a fixed point
of a new substitution defined on the set of the minimal balanced blocks. There
exist an algorithm to obtain this morphism (or substitution). Since $0$ is a point from $X_\sigma$ there exist two
minimal initial word $v_1$ and $v_2$ from the language of $\sigma_1$ and
$\sigma_2$ respectively  such that $l(v_1) = l(v_2)$

We denote the block  $\begin{pmatrix}
v_1\\
v_2\\
\end{pmatrix}$ and we consider $\sigma_1(v_1)$ and $\sigma_2(v_2)$ we obtain a second block  $\begin{pmatrix}
\sigma_1(v_1)\\
\sigma_2(v_2)\\
\end{pmatrix}$ with the property $l(\sigma_1(v_1))=l(\sigma_2(v_2))$ because $\sigma_1$ and $\sigma_2$ have
the same matrix. These blocks have a finite length, because the return time in $X_\sigma$ is bounded.
We consider the decomposition of this balanced  block $\begin{pmatrix}
\sigma_1(v_1)\\
\sigma_2(v_2)\\
\end{pmatrix}$  with minimal balanced  blocks.\\
This mean we can write $\begin{pmatrix}
\sigma_1(v_1)\\
\sigma_2(v_2)\\
\end{pmatrix} = \begin{pmatrix}
u_1\ldots u_k\\
w_1\ldots w_k\\
\end{pmatrix}$ with the property $l(u_1) = l(w_1)$, \ldots ,$l(u_n) = l(v_n)$.\\

With this method we obtain a finite numbers of blocks with the same
abelianization. We consider this set of blocks and we consider the image of
each block with the two substitutions $\sigma_1$ and $\sigma_2$ and we obtain a
morphism witch generate all the common points of the stepped lines.

\end{proof}

\section{Examples}
\subsection{Algorithm to obtain the morphism of the common points of two Rauzy fractals }

\subsubsection{Example 1}

 I will take the example of $\tau_1$ and $\tau_2$ to show how the algorithm is working. In this example the first minimal balanced block that we consider is the beginning of the two fixed points associated to $\tau_1$ and $\tau_2$ it will be $\begin{pmatrix}
a\\
a\\
\end{pmatrix}$.\\
And we consider the image of the first element of this block by $\tau_1$ and
the second one by $\tau_2$ so we obtain : $\begin{pmatrix}
a\\
a\\
\end{pmatrix}\overset{\tau_1, \tau_2}\longrightarrow\begin{pmatrix}
aba\\
aab\\
\end{pmatrix}$.\\
We denote by $A$ the minimal  balanced  block $\begin{pmatrix}
a\\
a\\
\end{pmatrix}$ and by $B$ the minimal balanced block $\begin{pmatrix}
ba\\
ab\\
\end{pmatrix}$.\\
So we obtain $A\rightarrow AB.$\\
The second step is to consider the same thing with the new block
$\begin{pmatrix}
ba\\
ab\\
\end{pmatrix}$.\\
We consider the image of this block with the two substitution $\tau_1$ and
$\tau_2$, and we obtain :
\begin{center}
$\begin{pmatrix}
ba\\
ab\\
\end{pmatrix}\overset{\tau_1, \tau_2}\longrightarrow\begin{pmatrix}
ababa\\
aabba\\
\end{pmatrix}$.\\
\end{center}
We obtain an other block $\begin{pmatrix}
b\\
b\\
\end{pmatrix}$
and we denote by $C$ the projection over this new block and we obtain the image
of $B$ is $ABCA$.
We continuous with this algorithm and we obtain the image of the block
$\begin{pmatrix}
b\\
b\\
\end{pmatrix}$ is the new block $\begin{pmatrix}
ab\\
ba\\
\end{pmatrix}$.
So we obtain the image of the letter $C$ is a new letter $D$.
Finally the image of the letter $D$ is $DAAC$.
So, we obtain an alphabet  $\mathcal{B}$ in $4$ letters and we can define the morphism $\phi$ as :

\begin{center}
$\phi:\left\{
\begin{array}{ll}
A\rightarrow AB\\
B\rightarrow ABCA\\
C\rightarrow D\\
D\rightarrow DAAC\\
\end{array}\right. $
\end{center}
And we consider the projection $\pi$ of the letters $A, B, C, D$ in the sets of
blocks $\begin{pmatrix}
a\\
a\\
\end{pmatrix}$,
$\begin{pmatrix}
ba\\
ab\\
\end{pmatrix}$, $\begin{pmatrix}
b\\
b\\
\end{pmatrix}$ et
$\begin{pmatrix}
ab\\
ba\\
\end{pmatrix}$\\
Then we have : $\begin{pmatrix}
\tau_1^n(a)\\
\tau_2^n(a)\\
\end{pmatrix}$
$ =  \pi(\phi^n(A))$.\\
The morphism $\phi$ generate all the common points of the two Rauzy fractals
associated to  $\tau_1$ and $\tau_2$.

\begin{figure}[h]
\begin{center}
\scalebox{0.45}{\includegraphics{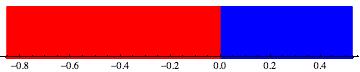}}
\scalebox{0.45}{\includegraphics{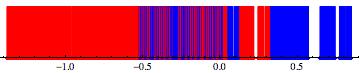}}
\caption{ The Rauzy fractals of $\tau_1$ and $\tau_2$.}
\end{center}
\end{figure}
\begin{figure}[h]
\begin{center}
\scalebox{0.45}{\includegraphics{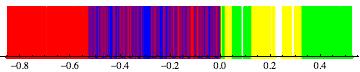}}
\caption{ Common points of  $\tau_1$ and $\tau_2$ with distinction of  blocs defined with $\phi$}
\end{center}
\end{figure}

\subsubsection{Example 2}
For the two substitutions of Tribonacci and the flipped Tribonacci it is more
complicated see Figure[7], we can define the morphism $\phi$ which generate all
the common points as follows:

\begin{center}
$\phi:\left\{
\begin{array}{ll}
A\rightarrow AB\\
B\rightarrow C\\
C\rightarrow AD\\
D\rightarrow AE\\
E\rightarrow F\\
F\rightarrow ADDGA\\
G\rightarrow AH\\
H\rightarrow ID\\
I\rightarrow ADJ\\
J\rightarrow AHK\\
K\rightarrow IDGA
\end{array}\right.$
\end{center}
and the projection map $\pi$ :

\begin{center}
$\pi:\left\{
\begin{array}{ll}
A\rightarrow \begin{pmatrix}
a\\
a\\
\end{pmatrix}\\
B\rightarrow \begin{pmatrix}
b\\
b\\
\end{pmatrix}\\
C\rightarrow \begin{pmatrix}
ac\\
ca\\
\end{pmatrix}\\
D\rightarrow \begin{pmatrix}
ba\\
ab\\
\end{pmatrix}\\
E\rightarrow \begin{pmatrix}
cab\\
bca\\
\end{pmatrix}\\
F\rightarrow \begin{pmatrix}
aabac\\
caaab\\
\end{pmatrix}\\
G\rightarrow \begin{pmatrix}
cab\\
abc\\
\end{pmatrix}\\
H\rightarrow\begin{pmatrix}
abac\\
bcaa\\
\end{pmatrix}\\
I\rightarrow \begin{pmatrix}
abaca\\
caaab\\
\end{pmatrix}\\
J\rightarrow\begin{pmatrix}
cabaab\\
ababca\\
\end{pmatrix}\\
K\rightarrow \begin{pmatrix}
ababac\\
bcaaab\\
\end{pmatrix}\end{array}\right.$

\end{center}

\begin{figure}[h]
\begin{center}
\scalebox{0.8}{\includegraphics{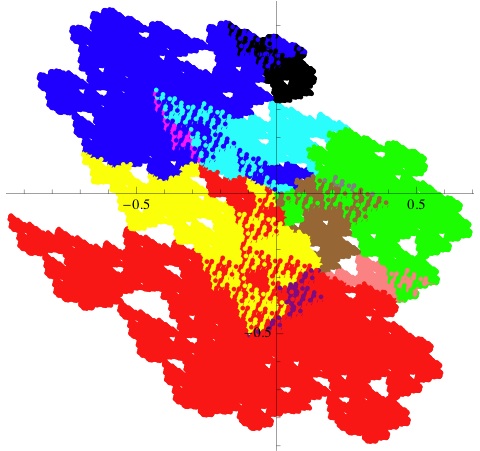}}
\caption{Sets of common points of the Tribonacci substitution and the flipped substitution. Each color stands for a
different letter of $\mathcal{B}$ and shows the dynamics of the morphism $\phi$.}

\end{center}
\end{figure}
\begin{corollaire}

Let $\sigma_1$ and $\sigma_2$ be the two substitution Tribonacci and the
flipped Tribonacci defined as follows:

\begin{center}
$\sigma_1:\left\{
\begin{array}{ll}
a\rightarrow ab\\
b\rightarrow ac\\
c\rightarrow a
\end{array}\right.$
\hspace{1cm}et\hspace{1cm} $\sigma_2:\left\{
\begin{array}{ll}
a\rightarrow ab\\
b\rightarrow ca\\
c\rightarrow a
\end{array}\right.$
\end{center}
\vspace{1cm}
 We consider $U$ and $V$ their two fixed points, then
the letter $c$ doesn't occur in the same position in $U$ and $V$.
\end{corollaire}
\begin{proof}
Minimal balanced blocks represents a decomposition of the two fixed
points $U $and $V $. We remark  that in these finite minimal  blocks there is no $c$ which appears
in the same position. One can then deduce that the letter $c$ does not appear
in the same position in two fixed points $U$ and $V$.
\end{proof}

\subsubsection{Example 3}

Now we will consider more general example defined as follows  :
\begin{center}
$\delta_i^1:\left\{
\begin{array}{ll}
a\rightarrow a^ib\\
b\rightarrow a^{i-1}c\\
c\rightarrow a
\end{array}\right.$
\hspace{1cm} and \hspace{1cm} $\delta_i^2:\left\{
\begin{array}{ll}
a\rightarrow aba^{i-1}\\
b\rightarrow aca^{i-2}\\
c\rightarrow a
\end{array}\right.$
\end{center}
$\delta_i^1$ and $\delta_i^2$ have the same incidence matrix. We can define the morphism of their common points for
all $i\geq 3$ as :
\begin{center}
\begin{figure}[h]
\begin{center}
\scalebox{0.45}{\includegraphics{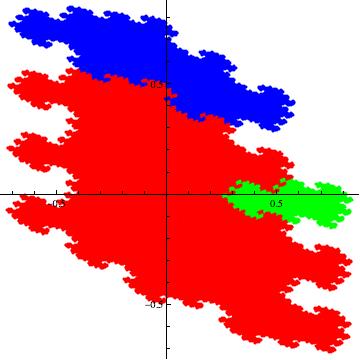}}
\scalebox{0.45}{\includegraphics{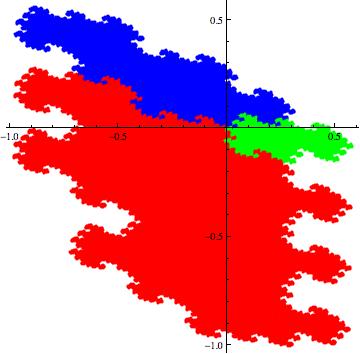}}
\caption{ The Rauzy fractals of $\delta_3^1$ and $\delta_3^2$.}
\end{center}
\end{figure}

$\phi_i:\left\{
\begin{array}{ll}
A\rightarrow AB\\
B\rightarrow AC\\
C\rightarrow (AAD)^{i-1}[AAE(AAD)^i]^{i-2}AAE(AAD)^{i-1}A\\
D\rightarrow AF\\
E\rightarrow (AAD)^{i-3}A\\
F\rightarrow (AAD)^{i-1}[AAE(AAD)^i]^{i-3}AAE(AAD)^{i-1}A.
\end{array}\right.$

$\pi_i:\left\{
\begin{array}{ll}
A\rightarrow \begin{pmatrix}
a\\
a\\
\end{pmatrix}\\
B\rightarrow \begin{pmatrix}
a^{i-1}b\\
ba^{i-1}\\
\end{pmatrix}\\
C\rightarrow \begin{pmatrix}
a^{i-1}b(a^{i}b)^{i-2}a^{i-1}c\\
ca^{i-1}(ba^{i})^{i-2}ba^{i-1}\\
\end{pmatrix}\\
D\rightarrow \begin{pmatrix}
a^{i-2}b\\
ba^{i-2}\\
\end{pmatrix}\\
E\rightarrow \begin{pmatrix}
a^{i-3}c\\
ca^{i-3}\\
\end{pmatrix}\\
F\rightarrow \begin{pmatrix}
a^{i-1}b(a^{i}b)^{i-3}a^{i-1}c\\
ca^{i-1}(ba^i)^{i-3}ba^{i-1}\\
\end{pmatrix}\\
\end{array}\right.$
\end{center}
\begin{figure}[h]
\begin{center}
\scalebox{0.5}{\includegraphics{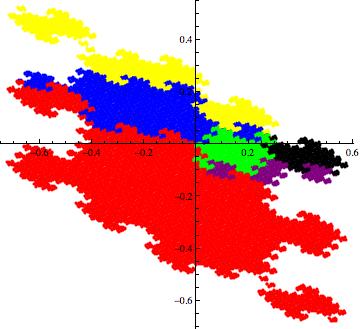}}
\caption{Sets of common points of $\delta_3^1$ and $\delta_3^2$.}

\end{center}
\end{figure}

\textbf{Remark.} The property $0$  is inner point is sufficient, and we have
this example of substitutions with the same incidence matrix but the
intersection is reduced to the origin.\\

We can give an example where the intersection is empty. We consider the two
substitutions $\chi_1$ and $\chi_2$ defined as follows :

 \begin{center}
$\chi_1:\left\{
\begin{array}{ll}
a\rightarrow aab\\
b\rightarrow ab\\
\end{array}\right.$
\hspace{1cm} and \hspace{1cm} $\chi_2:\left\{
\begin{array}{ll}
a\rightarrow baa\\
b\rightarrow ba\\
\end{array}\right.$
\end{center}

\begin{figure}[h]
\begin{center}
\scalebox{0.35}{\includegraphics{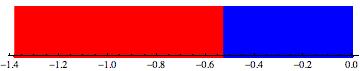}}
\scalebox{0.35}{\includegraphics{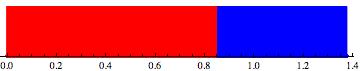}}
\caption{ The Rauzy fractals of $\chi_1$ and $\chi_2$.}
\end{center}
\end{figure}

\begin{proof}
We consider $u_1$ and $u_2$ the two fixed points associated to $\chi_1$ and
$\chi_2$ respectively.
If $a.x$ is a prefix of $u_1$ then $b.x$ is a prefix of $u_2$.\\
We will reason by induction :
for $x=a$ it is so verified for $n=1$.\\
We suppose now that $a.x$ is prefix of $u_1$ and $b.x$ is a prefix of $u_2$
with $|x|=n$.\\
$\chi_1(a.x)=aab\chi_1(x)$ is prefix of $u_1$,
$\chi_2(b.x).b=ba\chi_2(x)b$ is a prefix of $u_2$ if and only if
$\chi_1(x)=\chi_2(x)b$.\\
We have for the two letter $a$ and $b$:\\
\begin{itemize}
  \item $x=a$ : $b.\chi_1(a)=baab=\chi_2(a)b.$
  \item $x=b$ : $b.\chi_1(b)=bab=\chi_2(b)b.$
\end{itemize}
We consider now $x=x_1x_2\ldots x_n$ with $x_i\in\{a, b\}$\\
$b.\chi_1(x)=b.\chi_1(x_1x_2\ldots x_n)=b.\chi_1(x_1)\ldots\chi_1(x_n)$\\
\hspace{1.2cm}$=\chi_2(x_1).b.\chi_1(x_2)\ldots\chi_1(x_n)$\\
\hspace{1.2cm}\vdots\\
\hspace{1.2cm} $= \chi_2(x_1)\chi_2(x_2)\ldots\chi_2(x_n).b$\\
So we prove that there exist an infinite word $u$ such that $u_1=a.u$ and
$u_2=b.u$.
\end{proof}

\end{document}